\documentclass[10pt]{article}
\usepackage[all]{xy}
\usepackage{amsfonts,amsmath,oldgerm,amssymb,amscd}
\newcommand{\ra}{\rightarrow}

\newtheorem{theorem}{Theorem}[section]

\newtheorem{lemma}[theorem]{Lemma}

\newtheorem{corollary}[theorem]{Corollary}
\newtheorem{question}[theorem]{Question}

\newcommand{\gj}{\blacksquare}

\newcommand{\gt}{\theta}

\newcommand{\gL}{\Lambda}

\newcommand{\BQ}{\mbox{$\mathbb Q$}}

	\newcommand{\BZ}{\mbox{$\mathbb Z$}}

\newcommand{\CI}{\mbox{$\mathcal I$}}

	\newcommand{\p}{\mbox{$\mathfrak p$}}

\newcommand{\ot}{\mbox{\,$\otimes$\,}}	\newcommand{\op}{\mbox{$\oplus$}}

\begin{document}

\begin{center}
 {\Large \bf Ideal class groups of monoid algebras}\\\vspace{.2in}
 {\large  Husney Parvez Sarwar }\\
\vspace{.1in}
{\small
Department of Mathematics, IIT Bombay, Powai, Mumbai - 400076, India.\;\\
     parvez@math.iitb.ac.in, mathparvez@gmail.com}
\end{center}

\begin{abstract}
 Let $A\subset B$ be an extension of commutative reduced rings and $M\subset N$ an extension of positive commutative
 cancellative torsion-free monoids. 
 We prove that $A$ is subintegrally closed in $B$ and $M$ is 
 subintegrally closed in $N$ if and only if the group of invertible 
 $A$-submodules of $B$ is isomorphic to the group of invertible $A[M]$-submodules of $B[N]$ (\ref{6t2} $(b,d)$).
  In case $M=N$, we prove the same without the assumption that the ring extension is reduced (\ref{6t2}$(c,d)$).
   
\end{abstract}

\section{Introduction}
{\bf Throughout the paper we assume that all rings are commutative with unity, and all monoids are 
commutative cancellative torsion-free. }
For a ring extension $A\subset B$, the group of invertible $A$-submodule of $B$ is denoted
by $\CI(A,B)$. This group has been studied extensively by Roberts and Singh \cite{R-S}. Recently
 Sadhu and Singh (\cite{S-S}, Theorem 1.5) proved:
{\it
  Let $A\subset B$ be an extension of rings and $\BZ_+$ the monoid of positive
  integers. Then $A$ is subintegrally closed in 
 $B$ if and only if $\CI(A,B)\cong \CI(A[\BZ_+],B[\BZ_+])$.
}

Motivated by this result, we inquire the following
\begin{question}\label{6q1}
Let $A\subset B$ be an extension of rings and $M\subset N$ an extension of positive monoids. Are the following
  statements equivalent?

$(i)$  $A$ is subintegrally closed in $B$ and $M$ is subintegrally closed in $N$.

$(ii)$ $A[M]$ is subintegrally closed in $B[N]$.

  $(iii)$ $\CI(A,B)$ is isomorphic to $\CI(A[M],B[N])$.
 \end{question}
 It is always true that $(ii)\Rightarrow (i)$. If $B$ is a reduced ring, then $(i)\Rightarrow (ii)$ is (\cite{B-G}, Theorem 4.79).
 We answer the above question in the affirmative by proving the following result. Our proof uses
 Swan-Weibel's homotopy trick.

 \begin{theorem}\label{6t2}
 Let $A\subset B$ be an extension of  rings and $M\subset N$ an extension of positive monoids.

 $(a)$ If $A[M]$ is subintegrally closed in $B[N]$ and $N$ is affine, then $\CI(A,B)\cong \CI(A[M],B[N])$.
 
 $(b)$ If $B$ is reduced, $A$ is subintegrally closed in $B$ and $M$ is subintegrally closed in $N$,
 then $\CI(A,B)\cong \CI(A[M],B[N])$.
 
 $(c)$ If $M=N$, then the reduced condition on $B$ is not needed i.e. if $A$ is subintegrally closed in $B$, then
 $\CI(A,B)\cong \CI(A[M],B[M])$.
 
 $(d)$ (converse of $(a),(b)$ and $(c)$) 
 If $\CI(A,B)\cong \CI(A[M],B[N])$, then 
 $(i)$ $A$ is subintegrally closed in $B$, $(ii)$ $A[M]$ is subintegrally closed in $B[N]$, and 
 $(iii)$ $B$ is reduced or $M=N$. 
\end{theorem}

  The following result which is immediate from (\ref{6t2}), gives the exact conditions when $(i)\Rightarrow (ii)$ in 
  the above question (\ref{6q1}). 
 
 \begin{corollary}
   Let $A\subset B$ be an extension of rings and $M\subset N$ an extension of positive   monoids
  such that $A$ is subintegrally closed in $B$ and $M$ is subintegrally closed in $N$.
  
   $(i)$ If $B$ is reduced or $M=N$, then  $A[M]$ is subintegrally closed in $B[N]$. 
   
   $(ii)$ Conversely if $A[M]$ is subintegrally closed in $B[N]$ and $N$ is affine, then $B$ is reduced or $M=N$. 
\end{corollary}
 
  Let $A$ be a seminormal ring with $\BQ\subset A$  and $M$ a positive
  seminormal monoid. Then (\cite{B-G}, Theorem 8.42)
 proved that $Pic(A)\cong Pic(A[M])$. This result is due to Anderson (\cite{An82}, Theorem 1)
 in the case when $A[M]$ is an almost seminormal integral domain (see \cite{An82}, Definition). 
 As an application of our result (\ref{6t2}$c$), we deduce a special case of this result (see (\ref{6r1})).
 
 Sadhu and Singh (\cite{S-S}, Theorem 2.6) studied the relationship between the two groups $\CI(A,B)$
 and $\CI(A[\BZ_+],B[\BZ_+])$, when $A$ is not subintegrally closed in $B$. Using our result (\ref{6t2}),
 we generalize their result (\cite{S-S}, Theorem 2.6) to the monoid algebra situation in a straightforward way.

 \begin{theorem}\label{6mt2}
  Let $A\subset B$ be an extension of rings and let $^{^+}\!\!\!A$ denote the subintegral closure of
  $A$ in $B$. Assume that $M$ is a positive monoid. Then
  
  $(i)$ the following diagram
 \[
   \xymatrix
   { 1\ar@{->}[r] & \CI(A,{^{^+}\!\!\!A})\ar@{->}[r] \ar@{->}[d]^{\gt(A,^{^+}\!\!\!A)} & \CI(A,B)\ar@{->}[r]^{\phi(A,^{^+}\!\!\!A,B)}
     \ar@{->}[d]^{\gt(A,B)} & \CI(^{^+}\!\!\!A,B)\ar@{->}[d]^{\gt(^{^+}\!\!\!A,B)}_\simeq \ar@{->}[r] & 1\\
     1\ar@{->}[r] & \CI(A[M],{^{^+}\!\!\!A[M]})\ar@{->}[r] & \CI(A[M],B[M])\ar@{->}[r] & \CI(^{^+}\!\!\!A[M],B[M])\ar@{->}[r] & 1
   }
  \]
  is commutative with exact rows.
  
  $(ii)$ If $\BQ\subset A$, then $\CI(A[M],{^{^+}\!\!\!A[M]})\cong \BZ[M]\ot_{\BZ}\CI(A,{^{^+}\!\!\!A})$.
 \end{theorem}
 
 \section{Preliminaries}
 
 \begin{define}
$(1)$  Let $A\subset B$ be an extension of rings.
The extension $A\subset B$ is called {\it elementary subintegral} if $B=A[b]$ for some $b$ with $b^2,b^3\in A$.
If $B$ is a union of subrings which are obtained from $A$ by a finite succession of elementary 
subintegral extensions, then the extension $A\subset B$ is called {\it subintegral}. The 
{\it subintegral closure} of $A$ in $B$, denoted by $_B\!{^{^+}\!\!\!A}$, is the largest
subintegral extension of $A$ in $B$. We say $A$ is {\it subintegrally closed} in $B$ if $_B\!{^{^+}\!\!\!A}=A$.
A ring $A$ is called {\it seminormal} if it is reduced and subintegrally closed in $PQF(R):=\prod_{\p}QF(R/\p)$,
where  $\p$ runs through the minimal prime ideals of $R$ and $QF(R/\p)$ is the quotient field of $R/\p$
(see \cite{B-G}, pg 154).

$(2)$ Let $A\subset B$ and $A'\subset B'$ be two ring extensions. A morphism $\phi$ between the pairs
$(A,B)\ra (A',B')$ is a ring homomorphism $\phi: B\ra B'$ with $\phi(A)\subset A'$. For a ring 
extension $A\subset B$, if $\CI(A,B)$ denotes the multiplicative group of invertible $A$-submodules of $B$,
then $\CI$ is a functor from the category of ring extensions to the category of abelian groups. Let $\CI(\phi)$
denote the group homomorphism which is induced by the morphism $\phi$ of a ring extension.
If $B\subset B'$ and $A\subset A'$, then the inclusion map $i:B\ra B'$ defines a morphism of pairs 
$(A,B)\ra (A',B')$. We will denote $\CI(i)$ by $\gt(A,B)$. For basic facts pertaining
to ring extensions and the functor $\CI$, we refer the reader to \cite{R-S}.
  
 $(3)$ Let $M\subset N$ be an extension of monoids. 
  The extension $M\subset N$ is called {\it elementary subintegral} if $N=M\cup xM$ for some $x$ 
  with $x^2,x^3\in M$. If $N$ is a union of submonoids which are obtained from $M$ by a finite succession of elementary 
subintegral extensions, then the extension $M\subset N$ is called {\it subintegral}. The 
{\it subintegral closure} of $M$ in $N$, denoted by $_N\!{^{^+}\!\!M}$, is the largest
subintegral extension of $M$ in $N$. We say $M$ is {\it subintegrally closed} in $N$ if $_N\!{^{^+}\!\!M}=M$.
Let $gp(M)$ denote the group of fractions of the monoid $M$. We say $M$ is 
 {\it seminormal} if it is subintegrally closed in $gp(M)$.

$(4)$ For a monoid $M$, let $U(M)$ denote the group of units of $M$. 
  If $U(M)$ is a trivial group, then $M$ is called {\it positive}.
  If $M$ is finitely generated, then $M$ is called {\it affine} .

  For basic definitions and facts pertaining to monoids and monoid algebras, we refer the reader to (Ch.2, Ch.4 of \cite{B-G}).
\end{define}

\medskip

\noindent{\bf Notation:} For a ring $A$, $Pic(A)$ denotes the Picard group of $A$, $U(A)$ 
denotes the multiplicative group of units of $A$ and $nil(A)$ denotes the nil radical of $A$.

We note down some results for later use.

The following result which follows with a repeated applications of (\cite{S-S}, Corollary 1.6), is due to Sadhu and Singh.
\begin{lemma}\label{6l7}
 Let $A\subset B$ be an extension of rings. Then $A$ is subintegrally closed in $B$ if and only if
 $A[\BZ^r_+]$ is subintegrally closed in $B[\BZ^r_+]$ for any integer $r>0$.
\end{lemma}

 The following result is (\cite{B-G}, Theorem 4.79) by observing that $sn_{B}(A)$ (the seminormalization of $A$ in $B$)
 is same as $_B\!{^{^+}\!\!\!A}$ (the subintegral closure of $A$ in $B$) in our notation.
 \begin{lemma}\label{6l11}
  Let $A\subset B$ be an extension of reduced rings and $M\subset N$ an extension monoids.
  Then $_B\!{^{^+}\!\!\!A}[N\cap sn(M)]$ is the subintegral closure of $A[M]$ in $B[N]$,
  where $sn(M)$ is the seminormalization (subintegral closure) of $M$ in $gp(M)$.
 \end{lemma}

\section{Main Results}

The following result is motivated from (\cite{An78}, Lemma 5.7).
\begin{lemma}\label{6l2}
 Let $R=R_0\op R_1\op \cdots$ and $S=S_0\op S_1\op \cdots$ be two positively graded rings with
 $R\subset S$ and $R_i\subset S_i,\forall i\geq 0$. If the canonical map $\gt(R,S): {\CI}(R,S)\ra {\CI}(R[X],S[X])$
 is an isomorphism, then the canonical map $\gt(R_0,S_0): {\CI(R_0,S_0)}\ra {\CI}(R,S)$ is
 an isomorphism.
\end{lemma}
\begin{proof}
 This result uses Swan-Weibel's homotopy trick.
 Let $j:(R_0,S_0)\ra (R,S)$ be the inclusion map and $\pi:(R,S)\ra (R_0,S_0)$ the canonical surjection defined as
 $\pi(s_0+s_1+\cdots+s_r)=s_0$, where $s_o+s_1+\cdots+s_r \in S$. Then $\pi j=Id_{(R_0,S_0)}$. 
 Applying the functor $\CI$, we get that $\CI(\pi)\gt(R_0,S_0) =Id_{\CI(R_0,S_0)}$, where $\gt(R_0,S_0)=\CI(j)$.  Hence
  the canonical map $\gt(R_0,S_0)$ is injective. So we have to prove that
 $\gt(R_0,S_0)$ is surjective.
 
 Let $e_0,e_1:(R[X],S[X])\ra (R,S)$ be two evaluation maps defined as $X\ra 0,X\ra 1$ respectively and $i$ the
 inclusion map from $(R,S)\ra (R[X],S[X])$. Then we get that $e_0i=e_1i$.
 Let $w:(R,S)\ra (R[X],S[X])$ be a map defined as $w(s)=s_0+s_1X+\cdots+s_rX^r$, where $s=s_0+s_1+\cdots+s_r\in S$.
 It is easy to see that $w$ is a ring homomorphism from $S\ra S[X]$ and moreover $w$ is a morphism of ring extensions
 i.e. $w(R)\subset R[X]$.  It is easy to see that  $e_0w=j\pi\, \cdots (a)$. 
 
 Since $e_0i=e_1i=Id_{(R,S)}$, we get that $\CI(e_0)\gt(R,S)=\CI(e_1)\gt(R,S)=Id_{\CI(R,S)}$ (Recall that $\gt(R,S)=\CI(i)$).
 Therefore $\CI(e_0)$ and $\CI(e_1)$ are inverses of the
 canonical isomorphism $\gt(R,S)$. Hence $\CI(e_0)=\CI(e_1)$. By $(a)$,
  we have $\CI(e_0)\CI(w)=\gt(R_0,S_0)\CI(\pi)$. 
 Hence  $\CI(e_1)\CI(w)=\gt(R_0,S_0)\CI(\pi)$. Note that 
 $\CI(e_1)\CI(w)=Id_{\CI(R,S)}=\gt(R_0,S_0)\CI(\pi)$. Therefore we
 get that $\gt(R_0,S_0)$ is surjective. This completes the proof.
 $\hfill \gj$
\end{proof}

The following result is (\cite{B-G}, Theorem 4.79) when the ring extension is reduced. We use the 
same arguments as in (\cite{B-G}, Theorem 4.42, 4.79) to prove the following result. For an alternate proof
of the following result, see Remark (\ref{6r2}).
\begin{lemma}\label{6l1}
 Let $A\subset B$ be an extension of rings and $M$ an affine monoid.
 Assume that $A$ is subintegrally closed in $B$. Then $A[M]$ is subintegrally
 closed in $B[M]$.
\end{lemma}

\begin{proof}
 It is easy to see that $A[gp(M)]\cap B[M]=A[M]$. Hence it is enough to prove that $A[gp(M)]$ is
 subintegrally closed in $B[gp(M)]$. Since $M$ is affine, 
 $gp(M)\cong \BZ^r$ for some integer $r>0$. Hence we have to prove that $A[\BZ^r]$ is
 subintegrally closed in $B[\BZ^r]$. Since subintegrality commutes with localization
 (see \cite{B-G}, Theorem 4.75d), we have only to prove that $A[\BZ^r_+]$ is subintegrally
 closed in $B[\BZ^r_+]$. This is indeed the case because of (\ref{6l7}).
 $\hfill \gj$
\end{proof}

\subsection{Proof of the Theorem (\ref{6t2})}

 $(a)$  Since $N$ is positive affine, $N$ has a positive grading by (\cite{B-G}, Proposition 2.17f).
   Since $M$ is a submonoid of $N$, it has a positive grading induced from $N$.
  Therefore both $A[M]$ and $B[N]$ have positive gradings. Hence
   we can write $A[M]=A_0\op A_1\op\cdots $  and $B[N]=B_0\op B_1\op\cdots$
   with $A_0=A$, $B_0=B$. We define $R:=A[M]$, $S:=B[N]$ and $R_0:=A$,
   $S_0:=B$. By the hypothesis, $R$ is subintegrally closed in $S$, hence by (\cite{S-S}, Theorem 1.5), 
   $\CI(R,S)\cong \CI(R[X],S[X])$. Therefore by Lemma (\ref{6l2}), we get that $\CI(A,B)\cong \CI(A[M],B[N])$.
   
  $(b)$ First assume that $N$ is affine. Since $B$ is reduced, by (\ref{6l11}), the subintegral closure
  of $A[M]$ in $B[N]$ is $_B\!{^{^+}\!\!\!A}[N\cap sn(M)]$. Note that $sn(M)= _{gp(M)}\!\!\!{^{^+}\!\!M}$ in our
  notation. It is easy to see that  $_N\!{^{^+}\!\!M}=N\cap  _{gp(M)}\!\!\!{^{^+}\!\!M}$. By hypothesis, $_B\!{^{^+}\!\!\!A}=A$
  and $_N\!{^{^+}\!\!M}=M$. Hence $A[M]$ is subintegrally closed in $B[N]$.
    Therefore by $(a)$, we get that $\CI(A,B)\cong \CI(A[M],B[N])$. 
    
    Now the case when $N$ is not affine. Let $\gL:=\{ N_i: i\in I\}$ be the set  of all affine submonoids of $N$. 
    Then $\gL$ forms a directed set by defining
  $N_i\leq N_j$ if $N_i$ is a submonoid of $N_j$.  Let $M_i:=M\cap N_i$, where $N_i\in \gL$. Since $M$ is subintegrally closed
  in $N$, it is easy to see that $M_i$ is subintegrally closed in $N_i$.
  Then $N=\cup_{N_i\in \Lambda}N_i$ and  $M=\cup M_i$. If $N_i\leq N_j$, then there exist a morphism of ring
  extension $\phi_{ij}:(A[M_i],B[N_i])\ra (A[M_j],B[N_j])$ induced from the inclusion map $B[N_i]\ra B[N_j]$.
  Hence $(\{(A[M_i],B[N_i])\}_{N_i\in \gL},\{\phi_{ij}\}_{N_i\leq N_j})$
  forms a directed system in the category of ring extensions. Then the direct limit of this system is $((A[M],B[N]),\{\phi_i\})$,
  where $\phi_i:(A[M_i],B[N_i])\ra (A[M],B[N])$ i.e.  $\underrightarrow{lim}_{\gL}(A[M_i],B[N_i])=(A[M],B[N])$.
  
  Similarly as in the above paragraph one sees that 
  $(\CI(A[M_i],B[N_i])_{N_i\in \gL}),\{\CI(\phi_{ij})\}_{N_i\leq N_j})$
  forms a directed system in the category of abelian groups.
  
 We want to prove that $\underrightarrow{lim}_{\gL} (\CI(A[M_i],
 B[N_i]))\cong \CI (\underrightarrow{lim}_{\gL}(A[M_i],B[N_i]))=\CI(A[M],B[N])$. For each $N_i$,
 we have a map $\CI(j):\CI(A[M_i],B[N_i])\ra \CI(A[M],B[N])$ induced by the inclusion map $j:B[N_i]
 \ra B[N]$. Hence by the universal property of the direct limit, there exist a map 
 $$\phi: \underrightarrow{lim}_{\gL}(\CI(A[M_i],B[N_i]))\ra \CI(A[M],B[N]).$$ 
 We claim that $\phi$ is an isomorphism. For surjectivity, let $I\in \CI(A[M],B[N])$. Hence there exist
 $N_k\in \gL$ such that $I\in \CI(A[M_k],B[N_k])$. Taking the image
 of $I$ inside $\underrightarrow{lim}_{\gL}\CI(A[M_i],B[N_i])$, we obtain that $\phi$ is surjective.
  Since the natural inclusion $j:B\ra B[N_i]$ induces an isomorphism $\CI(j):\CI(A,B)\cong \CI
 (A[M_i],B[N_i])$ for each $N_i$, we obtain that $\CI(A,B)\cong \underrightarrow{lim}_{\gL}\CI(A[M_i],B[N_i])$. 
 Now it is easy to see that $\phi$ is injective. Therefore
 we get that $\CI(A,B)\cong \CI(A[M],B[N])$.
 
 $(c)$ As in $(b)$, we can assume that $M=N$ is affine.
  Then by (\ref{6l1}), $A[M]$ is subintegrally closed in $B[M]$. 
 Hence as in the proof of  $(b)$, we get that $\CI(A,B)\cong \CI(A[M],B[M])$.
 
 $(d)$ $(i)$
 To prove that $A$ is subintegrally closed in $B$, let $b\in B$ with $b^2,b^3\in A$. Let $m\in M$.
 Let $I:=(b^2,1-bm)$ and $J:=(b^2,1+bm)$ be two $A[M]$-submodules of $B[N]$. Note that $IJ\subset A[M]$ and
 $(1-bm)(1+bm)(1+b^2m^2)=1-b^4m^4\in IJ$. Hence $1=b^4m^4+1-b^4m^4\in IJ$ i.e. $IJ=A[M]$.
 Therefore $I\in \CI(A[M],B[N])$. Let $\pi$ be the natural surjection from $B[N]\ra B$
 sending $N\ra 0$.  Then $\CI(\pi)(I)=A$. By hypothesis $\CI(\pi)$ is an isomorphism, hence $I=A[M]$. Therefore $b\in A$.
 Hence $A$ is subintegrally closed in $B$.
 
 $(ii)$ Let $g\in B[N]$ such that $g^2,g^3\in A[M]$. Let  $I:=(g^2,1+g+g^2)$ and $J:=(g^2,1-g+g^2)$ be two $A[M]$-submodules
 of $B[N]$. Then $(1+g+g^2)(1-g+g^2)=(1+g^2+g^4)\in IJ\Rightarrow 1+g^2\in IJ\Rightarrow 1=g^4+(1+g^2)(1-g^2)\in IJ$.
 Note that $IJ\subset A[M]$, hence $IJ=A[M]$. Therefore $I\in \CI(A[M],B[N])$.  Let $\pi(g)=b\in B$ ($\pi$ is defined in $(i)$). 
 Then $\CI(\pi)(I)=(b^2,1-b+b^2)$. Since $g^2,g^3\in A[M]$, we obtain that $b^2,b^3\in A$. 
 But $A$ is subintegrally closed in $B$ by $(i)$, hence we get that $b\in A$. Hence
 $\CI(\pi)(I),\CI(\pi)(J)$ are contained in $A$. Therefore $\CI(\pi)(I)=A\Rightarrow I=A[M]$. Hence $g\in A[M]$.
 This proves that $A[M]$ is subintegrally closed in $B[N]$.

 $(iii)$ Let us look at the following commutative diagram
 \[
\xymatrix
{\CI(A,B) \ar@{->}[rrr]^{\phi_1}\ar@{->}[d]^{\phi_2}&& &\CI(A/nil(A),B/nil(B))\ar@{->}[d]^{\phi_3}\\
 \CI(A[M],B[N])\ar@{->}^{\phi_4}[rrr] & & & \CI(\frac{ A[M]}{nil(A)[M]},\frac {B[N]}{nil(B)[N]}),
}
\]
where $\phi_i$ are natural maps $\forall i$. 
 By $(i)$, we get that $A$ is subintegrally closed in $B$, hence by (\cite{S-S}, Lemma 1.2), $nil(B)\subset A$. 
 Hence $nil(B)=nil(A)$. Therefore by (\cite{R-S}, Proposition 2.6), we get that $\phi_1$ is an isomorphism. 
 Since $N$ is a cancellative
 torsion-free monoid, by (\cite{B-G}, Theorem 4.19), $nil(B[N])=nil(B)[N]$. By $(c)$, 
 we get that $\phi_3$ is an isomorphism. Hence $\phi_2$ is an isomorphism if
 and only if $\phi_4$ is an isomorphism. By (\cite{R-S}, Proposition 2.7),
 $\phi_4$ is an isomorphism if and only if $(1+nil(B)[N])/(1+nil(A)[M])$ is
 a trivial group. Since $nil(B)=nil(A)$ this is equivalent to $nil(B)=0$ (i.e. $B$ is reduced) or $M=N$.
 $\hfill \gj$

\begin{remark}\label{6r2}
 If $A$ is subintegrally closed in $B$ and $M=N$, then we obtain $\CI(A,B)\cong \CI(A[M],B[M])$ from the 
 arguments as in (\ref{6t2}$d(iii))$
 without using Lemma (\ref{6l1}). Hence
 using (\ref{6t2}$(d(ii))$), we get that $A[M]$ is subintegrally closed in $B[M]$. This gives an alternate 
 proof of Lemma (\ref{6l1}) without the hypothesis that $M$ is affine. $\hfill \gj$
\end{remark}

\begin{corollary}
 Let $A\subset B$ be an extension of reduced rings such that $A$ is subintegrally closed in $B$.
 Then $\CI(A,B)\cong \CI(A[X_1,\ldots,X_m],B[X_1,\ldots,X_m,Y_1,\ldots, Y_n])$.
\end{corollary}
\begin{proof}
 Observe that the submonoid generated by $(X_1,\ldots,X_m)$ is subintegrally closed in the
 monoid generated by $(X_1,\ldots,X_m,Y_1,\ldots, Y_n)$. Hence we obtain the result using (\ref{6t2}$(b)$).
 $\hfill \gj$
\end{proof}

In the following remark, we give an application of the result (\ref{6t2}$(c)$).
 \begin{remark}\label{6r1}(cf. \cite{S-S} Remark 1.8)
  Let $A$ be a seminormal ring which is Noetherian or an integral domain. 
  Let $M$ be a positive  seminormal monoid.
  Let $K$ be the total quotient ring of $A$. Then $K$ is a finite product of fields, 
   hence $Pic(K)$ is a trivial group. By Anderson (\cite{An88}, Corollary 2), 
   $Pic(K[M])$ is a trivial qroup. By (\cite{B-G}, Proposition 4.20), $U(K)=U(K[M])$
  and $U(A)=U(A[M])$. Now using the same arguments as in (\cite{S-S}, Remark 1.8), one can
  easily  deduce that $Pic(A)\cong Pic(A[M])$ from (\ref{6t2}).
 $\hfill \gj$
 \end{remark}
 
\medskip

\subsection{Proof of the Theorem (\ref{6mt2})}

$(i)$ Following the arguments of (\cite{S-S}, Theorem 2.6), we observe that we have only to prove that
the maps $\phi(A,{^{^+}\!\!\!A},B)$ and $\phi(A[M],{^{^+}\!\!\!A[M]},B[M])$ are surjective.
Since $^{^+}\!\!\!A$ is subintegrally closed in $B$, $\gt({^{^+}\!\!\!A},B)$ is surjective by (\ref{6t2}$c$).
Therefore we have only to show that $\phi(A,{^{^+}\!\!\!A},B)$ is surjective. But this follows from (\cite{Sa}, Proposition 3.1)
by taking $C={^{^+}\!\!\!A}$.

$(ii)$  If $A\subset B$ be a subintegral extension of $\BQ$-algebras, then a natural isomorphism $\xi_{B/A}: B/A\ra \CI(A,B)$ is
defined in \cite{R-S}. As in (\cite{R-S}, Lemma 5.3) this yield a commutative diagram
\[
\xymatrix
{^{^+}\!\!\!A/A \ar@{->}[r]^{\xi_{^{^+}\!\!\!A/A}}\ar@{->}[d]^{j} &\CI(A,{^{^+}\!\!\!A})\ar@{->}[d]^{\gt(A,{^{^+}\!\!\!A})}\\
 ^{^+}\!\!\!A[M]/A[M]\ar@{->}[r]^{\xi} & \CI(A[M],{^{^+}\!\!\!A[M]}),
}
\]
 where $\xi:=\xi_{^{^+}\!\!\!A[M]/A[M]}$. Both $\xi_{^{^+}\!\!\!A/A}$ and $\xi$
are isomorphisms by (\cite{R-S}. Main Theorem 5.6 and \cite{R-R-S}, Theorem 2.3).
Now $\CI(A[M],{^{^+}\!\!\!A[M]})\cong {^{^+}\!\!\!A[M]}/A[M]\cong \BZ[M]\ot_{\BZ}{ ^{^+}\!\!\!A}/A\cong \BZ[M]\ot_{\BZ}\CI(A,{^{^+}\!\!\!A})$.
$\hfill \gj$

\medskip
{\bf Acknowledgement:}
I am grateful to my advisor, Prof. Manoj K. Keshari, for his
guidance, motivation and many useful discussions. 
I thank Vivek Sadhu for introducing his results to me and for some useful discussions.
I acknowledge the financial support of  Council of Scientific and Industrial Research
(C.S.I.R.), Government of India. I sincerely thank the referee for his/her valuable
suggestions and for pointing out the mistakes in the earlier version.

\end{document}